\documentclass[12 pt]{amsart}
\usepackage{amsfonts}
\usepackage{enumerate}
\usepackage{amssymb}
\usepackage{times}
\usepackage{color}
\newtheorem{Exa}{Example}

\newtheorem{Th}{Theorem}[section]

\newtheorem{Def}[Th]{Definition}
\newtheorem{Rem}[Th]{\it Remark}
\newtheorem{Cor}[Th]{Corollary}
\newtheorem{La}[Th]{Lemma}
\newtheorem{Example}[Th]{Example}

\newcommand{\be}{\begin{eqnarray*}}
\newcommand{\ee}{\end{eqnarray*}}

\newcommand{\N}{\bf N}

\newcommand{\ra}{\rightarrow}

\newcommand{\IEJ}{\begin{Exa}}
\newcommand{\EEJ}{\end{Exa}}
\usepackage{setspace}

\newcommand{\T}{\mathbb T}

\newcommand{\thi}{\begin{Th}}
\newcommand{\thf}{\end{Th}}
\newcommand{\coi}{\begin{Cor}}
\newcommand{\cof}{\end{Cor}}
\newcommand{\dfi}{\begin{Def}}
\newcommand{\dff}{\end{Def}}

\newcommand{\isoesa}{${\rm iso}\,\mathfrak{S}_a(\mathbb{T})$}

\newcommand{\esuf}{$\mathfrak{S}_{uf}(\mathbb{T})$} 
\newcommand{\esubf}{$\mathfrak{S}_{ubf}(\mathbb{T})$}

\newcommand{\esuw}{$\mathfrak{S}_{uw}(\mathbb{T})$} 

\newcommand{\esubw}{$ \mathfrak{S}_{ubw}(\mathbb{T})$}

\title{The property $(E_A)$ and   local spectral theory}

 \author{Elvis Aponte\thinspace$^{*,1}$,  Lourival Lima\thinspace$^{2}$ Jos\'e Sanabria\thinspace$^{3}$}
 \address{$^{1, 2}$ Departamento de Matem\'aticas, Facultad de Ciencias Naturales y Matem\'aticas, Escuela Superior Polit\'ecnica del Litoral (ESPOL), Campus Gustavo Galindo km. 30.5  V\'ia Perimetral, \mbox{Guayaquil EC090112, Ecuador.}}
\email{$^{1}$ ecaponte@espol.edu.ec\qquad ORCID: 0000-0002-5745-3233\newline $^{2}$ lourodri@espol.edu.ec\qquad ORCID: 0000-0001-6826-3315}

\address{$^{3}$ Departamento de Matem\'aticas, {Facultad} de Educaci\'on y Ciencias,  Universidad de Sucre, \mbox{Sincelejo, Colombia.}}
\email{$^{3}$  jesanabri@gmail.com\qquad ORCID: 0000-0002-9749-4099}

 \keywords{  Property $(E_A)$, property $(bz)$ semi-Fredholm operator,  SVEP. \newline \indent $^{*}$Corresponding author}

\date{}

\begin{document}
\maketitle
\normalsize

\begin{abstract}

In this paper, we introduce and study the spectral property  $(E_A)$. This property means that the difference between the approximate point spectrum and the upper semi-Fredholm spectrum coincides with the difference between the approximate point spectrum and the upper semi-Weyl spectrum. Together with local spectral theory, we explore the behavior of this property under certain topological conditions and derive characterizations for the operators that verify it. Furthermore, we establish sufficient conditions that guarantee that a bounded linear operator verifies the property $(E_A)$.

\end{abstract}
\section{ Introduction}

Local spectral theory of linear operators has applications in diverse domains of science; for example, in artificial intelligence, it is very valuable for the analysis of clustering and dimensionality reduction algorithms, as pointed out by \cite{Saul}. Additionally, in the field of massive data analysis, the study of eigenvalues becomes crucial when describing matrices, which, in turn, represent linear operators \cite{Chen}.
Among the research lines that have allowed a great development of local spectral theory are the studies relating the classical Weyl-type theorems \cite{weyl}-\cite{Cob}  and also Browder-type theorems \cite{Han}-\cite{AGP2}. These theorems are spectral properties that classify linear operators and are based on parts of the spectrum for their description.

Recently, studies have been carried out on the spectral properties that facilitate the analysis of the spectrum of a linear operator. For instance, property $(gaz)$  was studied in \cite{AAG} to describe different parts of the spectrum of an operator and in \cite{AJQV} to analyze the spectrum of a tensor product of two operators, this property is derived from the equality between the upper semi $B$-Weyl spectrum and the left Drazin invertible spectrum. Property $(V_{\Pi})$, which is a variant of the classical Browder's theorem, was explored in \cite{uuu} to describe parts of the spectrum and the quasi-nilpotent part of linear operators from a topological point of view. Property $(Bv)$ was considered in \cite{Tensor2022} to establish the conditions under which different parts of the spectrum of a tensor product operator coincide. There are more than forty spectral properties defined in the above style to describe parts of the spectrum of an operator as can be seen in reference \cite{SANA0b}. Although in the aforementioned studies, the operators involved are semi-Fredholm, the upper semi-Fredholm spectrum was not analyzed in these studies.  This shortcoming has been partially covered by property $(bz)$ that has been explored in references \cite{Ouidren222}-\cite{new2}, and it is equivalent to the equality between the upper semi-Fredholm spectrum and the upper semi-Browder spectrum. 

It is interesting to explore whether the local spectral theory yields new and different results when the upper semi-Fredholm spectrum is related to Weyl-type spectra, which has not been documented in the literature.  This leads to the enrichment of the theory of Fredholm-type operators.

The purpose of this manuscript is to explore the equality between the upper semi-Fredholm spectrum and the upper semi-Weyl spectrum of bounded linear operators. This equality is equivalent to analyzing a new spectral property, called property $(E_A)$.  We will use local spectral theory methods to obtain new descriptions of the upper semi-Fredholm spectrum. For this, we developed  three sections as follows:

In Section 2, we give the terminology to use and some known results.

In Section 3, we introduce the property $(E_A)$ and its generalized version, show that they are equivalent, and give some characterizations that allow us to see what the upper semi-Fredholm spectrum looks like.

In Section 4, we use some notions of topology and local spectral theory to give sufficient conditions for a bounded linear operator to verify property $(E_A)$.
 
\section{Preliminaries}

In this section, we will delineate the terminology employed throughout the article. Much of this terminology is derived from Aiena's book \cite{Aiena}. The symbol $\mathbb{C}$ denotes the space of complex numbers and $\mathbb{L}(\mathbb{X})$ represents the Banach algebra of all bounded linear operators on a complex Banach space $\mathbb{X}$ in itself. The set of accumulation points of $B\subseteq \mathbb{C}$ is denoted by $Acc(B)$.

 For $\mathbb{T} \in \mathbb{L}(\mathbb{X})$, $\alpha(\mathbb{T})$ denotes the dimension of $\text{ker}(\mathbb{T})$ (the kernel of $\mathbb{T}$), and $\beta(\mathbb{T})$ represents the co-dimension of $\mathbb{T}(\mathbb{X})$ (the range of $\mathbb{T}$). We  denoted by $p(T)$ the  \textit{ascent} of $\mathbb{T}$ and is defined as the smallest non-negative integer such that $\text{ker}(\mathbb{T}^{p(\mathbb{T})})= \text{ker}( \mathbb{T}^{p(\mathbb{T})+1})$. Observe that in some cases it is possible that $p(\mathbb{T})= \infty$. Similarly,  the \textit{descent} of $\mathbb{T}$ is represented by $q(\mathbb{T})$ and is defined as the smallest non-negative integer such that $\mathbb{T}^{q(\mathbb{T})}(X)= \mathbb{T}^{q(\mathbb{T})+1}(X)$. Occasionally, it may happen that $q(\mathbb{T})= \infty$.

In the following points, we look at the definition of some types of operators that are classics in the literature:
\begin{enumerate}
\item[(i)] An operator $\mathbb{T} \in \mathbb{L}(\mathbb{X})$ is said to be an \textit{upper semi-Fredholm}, if it has the closed range and $\alpha (\mathbb{T})< \infty$. In contrast, \textit{lower semi-Fredholm operators} are characterized by having $\beta (\mathbb{T})< \infty$.  Moreover, the class of \textit{Fredholm operators} is one whose elements are upper semi-Fredholm operators and at the same time are lower semi-Fredholm operators. 
\item[(ii)] The  \textit{index} of an operator $\mathbb{T} \in \mathbb{L}(\mathbb{X})$ is define by $ind(\mathbb{T}):=\alpha(\mathbb{T})-\beta(\mathbb{T})$. Thus, $\mathbb{T}$ is an \textit{upper semi-Weyl operator}, if it is an upper semi-Fredholm operator with $ind(\mathbb{T})\leq 0$,  and $\mathbb{T}$ is a \textit{lower semi-Weyl operator}, if it is a lower semi-Fredholm operator with $ind(\mathbb{T}) \geq 0$. Also, $\mathbb{T}$ is a \textit{Weyl operator}, if  $ind(\mathbb{T}) = 0$.
\item[(iii)] The concept of a Browder operator is centered on the ascent or descent of an operator $\mathbb{T}$. If $\mathbb{T} \in \mathbb{L}(\mathbb{X})$ has finite ascent and is upper semi-Fredholm, then it is classified as an \textit{upper semi-Browder operator}. On the other hand, if $\mathbb{T}$ shows finite descent and is lower semi-Fredholm, then it is termed a \textit{lower semi-Browder operator}. A \textit{Fredholm operator} $\mathbb{T}$ is designated as a \textit{Browder operator} when it exhibits both finite ascent and descent.
\item[(iv)] The operator $\mathbb{T}$ is said to be an  \textit{left Drazin invertible}, if $p(\mathbb{T}) <  \infty$ and  $\mathbb{T}^{p(\mathbb{T})+1}(X)$ is closed.
 \item[(v)]  $\mathbb{T} \in \mathbb{L}(\mathbb{X})$ is characterized as a \textit{$B$-Fredholm} (resp., \textit{upper semi $B$-Fred\-holm, lower semi $B$-Fredholm}) \textit{operator}, if for some integer $n\geq 0$, the range $\mathbb{T}^n (\mathbb{X})$ is closed, and $\mathbb{T}_{[n]}:= \mathbb{T}_{\mathbb{T}^n(\mathbb{X})}$ (the restriction of $\mathbb{T}$ to $\mathbb{T}^n(\mathbb{X})$) is a Fredholm (resp., upper semi-Fredholm, lower semi-Fredholm)  operator. We put $\mathbb{T}_{[0]} = \mathbb{T}$. The index of a $B$-Fredholm type operator is defined as $\text{ind} \ \mathbb{T} := \text{ind}\ \mathbb{T}_{[n]}$.  
 \item[(vi)]  $\mathbb{T} \in \mathbb{L}(\mathbb{X})$ is classified as a \textit{B-Weyl} (resp., \textit{upper semi B-Weyl}, \textit{lower semi B-Weyl}) \textit{operator}, if for some integer $n\geq 0$, $\mathbb{T}^n (\mathbb{X})$ is closed, and $\mathbb{T}_{[n]}$ is a Weyl (resp., upper semi-Weyl, lower semi-Weyl)  operator.
\end{enumerate}

Many classes of operators $\mathbb{T} \in  \mathbb{L}(\mathbb{X})$ can be characterized by their spectra. The spectrum of an operator $\mathbb{T}$ is defined as the set of complex numbers $\leftthreetimes$ for which the operator $\leftthreetimes \mathbb{I} - \mathbb{T}$ is not invertible in $\mathbb{L}(\mathbb{X})$. There are other spectra contained in the spectrum of an operator, for instance, if $\leftthreetimes$ is a scalar within the surjective spectrum of $\mathbb{T}$, then $\leftthreetimes \mathbb{I} - \mathbb{T}$ is not a member of the class of surjective operators in $\mathbb{L}(\mathbb{X})$.  Similarly, if $\leftthreetimes$ is a scalar within the punctual approximate spectrum of $\mathbb{T}$, then $\leftthreetimes \mathbb{I} - \mathbb{T}$ does not belong to the class of operators bounded below in $\mathbb{L}(\mathbb{X})$. It is important to mention that the bounded below operators in $\mathbb{L}(\mathbb{X})$ are those that have closed range and are injective. In the same way, for  $\mathbb{T} \in \mathbb{L}(\mathbb{X})$ different spectra are denoted as follows:
\begin{itemize} 
\item   $\mathfrak{S}({\mathbb{T}})$ as the spectrum of $\mathbb{T}$.
\item $\mathfrak{S}_{s}({\mathbb{T}})$ as the spectrum surjetive of $\mathbb{T}$.
\item $\mathfrak{S}_{a}({\mathbb{T}})$ as the approximate point spectrum. 
\item $\mathfrak{S}_{uf}({\mathbb{T}})$ as the upper semi-Fredholm spectrum.
\item $\mathfrak{S}_{lf}({\mathbb{T}})$ as the lower semi-Fredholm spectrum.
\item $\mathfrak{S}_{uw}({\mathbb{T}})$ as the upper semi-Weyl spectrum.
\item $\mathfrak{S}_{lw}({\mathbb{T}})$ as the lower semi-Weyl spectrum.
\item $\mathfrak{S}_{ub}({\mathbb{T}})$ as the upper semi-Browder spectrum. 
\item $\mathfrak{S}_{ld}({\mathbb{T}})$ as the left Drazin invertible spectrum.
\item  $\mathfrak{S}_{ubf}({\mathbb{T}})$ as the upper semi $B$-Fredholm spectrum.  
\item $\mathfrak{S}_{ubw}({\mathbb{T}})$ as the upper semi $B$-Weyl spectrum.
\end{itemize}
 
The classical {\em dual operator} of  $\mathbb{T} \in \mathbb{L}(\mathbb{X})$, over the space $\mathbb{X}^\ast: = \mathbb{L}(\mathbb{X}, \mathbb C)$ is defined by
$$
(\mathbb{T}^\ast f)(\mathbb{X}): = f(\mathbb{T}x) \quad \mbox{for all}\ x \in \mathbb{X} ,\, f \in \mathbb{X}^\ast.
$$

On the other hand, in \cite{Fi}, it is introduced the idea of the single-valued extension property, abbreviated as SVEP, is defined as the following: an operator  $\mathbb{T}\in \mathbb{L}(\mathbb{X})$ possesses the SVEP   at $\leftthreetimes_0 \in {\mathbb C}$,  if for all open disc $\mathbb D$ of  $\leftthreetimes_0$, the only analytic function $f: \mathbb D  \ra \mathbb{X} $ which satisfies the equation $(\leftthreetimes \mathbb{I} - \mathbb{T}) f(\leftthreetimes ) = 0$, for every $\leftthreetimes \in \mathbb D$, is the function $f \equiv 0$. It is said that  $\mathbb{T}$ possesses the SVEP if \,$\mathbb{T}$ possesses the SVEP for all $\leftthreetimes \in {\mathbb C}$.
 
 By   \cite[Theorem 3.8]{Aiena}, it turn out that if $p(\leftthreetimes \mathbb{I}-\mathbb{T})< \infty$, then $ \mathbb{T}$ possesses the SVEP at $\leftthreetimes$.
 
It is straightforward that if $\mathfrak{S}_{\rm a} (\mathbb{T})$ does not cluster at $\leftthreetimes$, then $\mathbb{T}$ possesses the SVEP at $\leftthreetimes$.

\begin{Rem} \label{Re1} {\rm  Both affirmations above are equivalent, when 
  $\leftthreetimes \mathbb{I } - \mathbb{T }$ is an operator of  Fredholm type. See \cite{AN}.  
  Note that $ \mathbb{T}$ possesses the SVEP in every isolated point of the spectrum, also in $\rho(\mathbb{T})={\mathbb C}\setminus \mathfrak{S}(\mathbb{T})$.
}
\end{Rem}

\section{The property $(E_A)$}

This section introduces property $(E_A)$ and its generalized version $(gE_A)$. In reality, these two properties are equivalent to each other, and classify the linear operators with the equality between the upper semi-Fredhlom spectrum and the upper semi-Weyl spectrum, a classification from which other relations and spectral properties are obtained. To get this, with   $\mathbb{T }\in \mathbb{L }(\mathbb{X})$, we define: 
 \begin{itemize}

\item $\delta_{uf}(\mathbb{T}):=\mathfrak{S}_a(\mathbb{T})\setminus \mathfrak{S}_{uf}(\mathbb{T})$.

 \item $\Delta_{uf}(\mathbb{T}):=\mathfrak{S}_a(\mathbb{T})\setminus \mathfrak{S}_{uw}(\mathbb{T})$.
 
 \item $\delta_{ubf}(\mathbb{T}):=\mathfrak{S}_a(\mathbb{T})\setminus \mathfrak{S}_{ubf}(\mathbb{T})$.

 \item $\Delta_{ubf}(\mathbb{T}):=\mathfrak{S}_a(\mathbb{T})\setminus \mathfrak{S}_{ubw}(\mathbb{T})$.

 \item $\Delta_{uw}(\mathbb{T}):=\mathfrak{S}(\mathbb{T})\setminus \mathfrak{S}_{uw}(\mathbb{T})$.
 
 \item $\mathcal{P}^{a}_{00}(\mathbb{T}):=\mathfrak{S}_a(\mathbb{T})\setminus \mathfrak{S}_{ub}(\mathbb{T})$.


      \item $\rho_{uf}(\mathbb{T}):= \mathfrak{S}(\mathbb{T})\setminus \mathfrak{S}_{uf}(\mathbb{T})$. 
\end{itemize}

 \begin{Example}
 Consider \(\mathbb{X} = \ell^2( \mathbb{N})\)  and let $\T\in \mathbb{X}$ be an operator defined by:
 \[
 \T(x_1, x_2, \dots) = (0, x_2, x_3, \dots ).
 \]
 Then  \(\mathfrak{S}_{a}(\T) = \{0, 1\}\) and \(\mathfrak{S}_{uf}(\T) = \{1\}\). Since
\(\alpha(\T) = 1\) and  \(\beta(\T) =  1\), it follows that \(0 \notin \mathfrak{S}_{uw}(\T )\). By the fact that \(\mathfrak{S}_{uf}(\T ) \subseteq \mathfrak{S}_{uw}(\T ) \subseteq \mathfrak{S}_{a}(\T )\), it turns out that  \(\mathfrak{S}_{uf}(\T ) = \mathfrak{S}_{uw}(\T )\) and thus, \(\delta_{uf}(\T) = \Delta_{uf}(\T)\).
\end{Example}

\begin{Example}
    Consider the zero operator $\T =\mathbf{0}$. Clearly, \(\leftthreetimes I - \T \) is invertible if and only if \(\leftthreetimes \neq 0\), so it is straightforward that \(\mathfrak{S}_{uf}(\T ) = \mathfrak{S}_{uw}(\T )=\mathfrak{S}_{w}(\T ) = \mathfrak{S}_{a}(\T ) = \{0\}\). On the other hand, we know that \(\mathfrak{S}_{ubf}(\T )\subseteq\mathfrak{S}_{ubw}(\T ) \subseteq \mathfrak{S}_{uf}(\T )\), and for any natural number \(n\), it holds that \(\mathbf{0}^n = \{0\}\), hence \(\alpha(\mathbf{0}_{[n]}) = \beta(\mathbf{0}_{[n]}) = 0\), whereby    
     \(\mathfrak{S}_{ubf}(\T ) = \mathfrak{S}_{ubw}(\T ) = \emptyset\), and therefore, \(\delta_{ubf}(\T) = \Delta_{ubf}(\T)\).
\end{Example}

These examples inspire us to investigate operators with similar spectral characteristics, for which we introduce the following definitions.

\begin{Def}   
Let  $\mathbb{T} \in \mathbb{L }(\mathbb{X})$. The operator $\mathbb{T}$ is said to verify
   property $(E_A)$ if  $\delta_{uf}(\mathbb{T})= \Delta_{uf}(\mathbb{T})$. It is said to verify property $(gE_A)$ if  $\delta_{ubf}(\mathbb{T})= \Delta_{ubf}(\mathbb{T})$.
 \end{Def}
 
Let $\mathbb{T} \in \mathbb{L }(\mathbb{X})$  possessing the SVEP for all $\leftthreetimes \in \delta_{uf}(\mathbb{T})$, so by Remark \ref{Re1}, for all $\leftthreetimes \in \delta_{uf}(\mathbb{T})$, it turn out that $p(\leftthreetimes I -\mathbb{T}) < \infty$, and so $ind(\leftthreetimes I -\mathbb{T}) \leq 0$,  whereby   $\delta_{uf}(\mathbb{T}) \subseteq \Delta_{uf}(\mathbb{T})$, hence $\mathbb{T}$  verifies  property $(E_A)$. Note that, the above also occurs if $T$ possesses the SVEP or satisfies any condition that implies that $T$ possesses the SVEP, for example, if $\mathfrak{S}_{a}(\T ) = \partial(\mathfrak{S}(\T ))$ (see \cite[Theorem 2.94]{Aiena}).

\begin{Example}\label{ejem11}
 Recall that the Ces\`{a}ro matrix $C$ is a lower triangular matrix such that the nonzero entries of the $n$-th row are $n^{-1}$, with $n \in \mathbb{N}$. Namely; 
$$
\left(\begin{array}{ccccc}
1 & 0 & 0 & 0 & \cdots \\
1 / 2 & 1 / 2 & 0 & 0 & \cdots \\
1 / 3 & 1 / 3 & 1 / 3 & 0 & \cdots \\
1 / 4 & 1 / 4 & 1 / 4 & 1 / 4 & \cdots \\
\vdots & \vdots & \vdots & \vdots & \vdots
\end{array}\right) .
$$
Consider $C$ as the operator $C_p$ acting on $\ell_p (\mathbb{N}) $, where $1<p<\infty$. In the reference \cite{roa}, Rhoades and Neumann proved that $\mathfrak{S}\left(C_p\right)$ is the closed disc $\Gamma_q$,  where $1 / p+1 / q=1$,  and
$$
\Gamma_q:=\{\leftthreetimes \in \mathbb{C}:|\leftthreetimes-q / 2| \leq q / 2\}.
$$
In the reference \cite{mg}, Gonz\'alez proved that for each $\mu \in \operatorname{int} \Gamma_q$, it turns out that  $\mu I-C_p$ is an injective Fredholm operator and hence, a bounded below operator. But, by \cite[Theorem 1.12]{Aiena}, $\mathfrak{S}_{a}(C_p )$  contain the
boundary of $\mathfrak{S}\left(C_p\right)$, hence  
$\mathfrak{S}_{a}(C_p ) = \partial(\Gamma_q)$.
Consequently, the operator $C_p$ verifies property  $(E_A)$, whereby 

$$\partial(\Gamma_q) = \mathfrak{S}_{uf}(C_p ) = \mathfrak{S}_{uw}(C_p).$$
\end{Example}
 
Next, we will see a characterization that will allow us to establish, among other things, that properties $(E_A)$ and $(gE_A)$ are equivalent.
    
 \begin{Th} \label{T1.6}
 $\mathbb{T} \in \mathbb{L }(\mathbb{X})$  verifies
     property $(E_A)$,  if and only if,   $\mathfrak{S}_{uf}(\mathbb{T}) = \mathfrak{S}_{uw}(\mathbb{T})$.
     \end{Th}
    \begin{proof}
        It always holds that \(\mathfrak{S}_{uf}(\T) \subseteq \mathfrak{S}_{uw}(\T)\). Now, if we suppose that \(\leftthreetimes \notin \ \mathfrak{S}_{uf}(\T)\), then  \(\leftthreetimes I - \T \) is an upper semi-Fredholm operator. There are two possibilities: one is that if \(\leftthreetimes \in \mathfrak{S}_a (\T)\), then \(\leftthreetimes \notin \mathfrak{S}_{uw}(\T)\), since \(\T\) satisfies the property \((E_A)\); other is that if \(\leftthreetimes \notin \mathfrak{S}_a (\T)\), it holds that \(\leftthreetimes \notin \mathfrak{S}_{uw}(\T)\), since \(\mathfrak{S}_{uw}(\T) \subseteq \mathfrak{S}_a(\T)\), so in any case,   \(\leftthreetimes \notin \mathfrak{S}_{uw}(\T)\). Hence,  we conclude that \( \mathfrak{S}_{uf}(\T) = \mathfrak{S}_{uw}(\T)\). The converse implication is straightforward.
    \end{proof}

 Next, we present an example of an operator that verifies property $(E_A)$, but its dual operator does not verify it.
\begin{Example}\label{e3.5}
Consider \(\mathbb{X} = \ell^2(\mathbb{N})\) and let $\T$ be the {unilateral left shift} defined as: $$\T(x_1,x_2,x_3,...)=(x_2,x_3,...), \text{  for all } (x_n)\in \ell^2(\mathbb{N}).$$
To this operator is known that the upper semi-Weyl spectrum is the unitary disc $\mathbf{D}(0,1)$. On the other hand, its upper semi-Fredholm spectrum is given by the unitary circle $\Gamma(0,1)$, see \cite[Example 2.3]{Ouidren222}. Therefore, the unilateral left shift does not verify property $(E_A)$. Now, consider its dual operator on $\ell^2(\mathbb{N})$ defined as:  $$\T^{*}(x_1,x_2,x_3,...)=(0,x_1,x_2,x_3,...), \text{  for all } (x_n)\in \ell^2(\mathbb{N}).$$ 
This operator is called the unilateral right shift operator. Note that $\T^{*}$ possesses the $SVEP$, and hence, $\T^{*}$ verifies property $(E_A)$. This ultimately allows establish that $\mathfrak{S}_{uf}(\T^{*}) = \mathfrak{S}_{uw}(\T^{*})= \Gamma(0,1)$.
\end{Example}

Given this last example, we see that property $(E_A)$ must be studied separately, both for an operator $\mathbb{T} \in \mathbb{L }(\mathbb{X})$, and for its dual $\mathbb{T^*}$. 

\begin{Cor}
Let $\mathbb{T} \in \mathbb{L }(\mathbb{X})$. Then, $\mathbb{T^*}$ verifies
     property $(E_A)$,  if and only if,   $\mathfrak{S}_{lf}(\mathbb{T}) = \mathfrak{S}_{lw}(\mathbb{T})$.
\end{Cor}
   \begin{proof}
   It follows by duality and by Theorem \ref{T1.6}.
   \end{proof}

  We require the following lemma to establish equivalence between the properties $(E_A)$ and $(gE_A)$.

\begin{La}\label{T1.44} Let $\mathbb{T }\in \mathbb{L }(\mathbb{X})$. Then,  
$\mathfrak{S}_{uf} 
 (\mathbb{T})=\mathfrak{S}_{uw}(\mathbb{T}) $, if and only if, $ \mathfrak{S}_{ubf}(\mathbb{T})=\mathfrak{S}_{ubw}(\mathbb{T}).$
\end{La} 

 \begin{proof}
  Suppose that \(\mathfrak{S}_{ubf}(\T ) = \mathfrak{S}_{ubw}(\T )\) and let \(\leftthreetimes_0 \notin \mathfrak{S}_{uf}(\T )\). Then \(\leftthreetimes_0 I - T\) is an upper semi-Fredholm operator and consequently, it is upper semi-B-Fredholm. By \cite[Theorem 1.117]{Aiena}  there exists an open disc \(\mathbb{D} \) centered at \(\leftthreetimes_0\) such that for all $\leftthreetimes \in \mathbb{D} \setminus\{\leftthreetimes_0\}$, the operator \(\leftthreetimes I - \T\) is  an upper semi-Fredholm with  \(\text{ind}(\leftthreetimes I - \T) = \text{ind}(\leftthreetimes_0 I - \T)\). Since \(\leftthreetimes_0 \notin \mathfrak{S}_{ubw}(\T )\), it follows that \(\leftthreetimes \notin \mathfrak{S}_{uw}(\T)\) and  therefore, \(\mathfrak{S}_{uw}(\T ) \subseteq\mathfrak{S}_{uf}(\T )\). 

  On the other hand, if \(\leftthreetimes_0 \notin \mathfrak{S}_{ubf}(\T )\), then \(\leftthreetimes_0I - \T\) is an upper semi-B-Fredholm  operator, thus, by \cite[Theorem 1.117]{Aiena}, there exists at least one open disc \(\mathbb{D}\) centered at \(\leftthreetimes_0\) such that for all \(\leftthreetimes \neq \leftthreetimes_0\) on \(\mathbb{D}\),  \(\leftthreetimes I - \T\) is upper semi-Fredholm and \(\text{ind}(\leftthreetimes I - \T ) = \text{ind} (\leftthreetimes_0 I - \T)\), so that \(\leftthreetimes I -\T\)  is an upper semi-Weyl operator. Hence \(\text{ind}(\leftthreetimes_0 I - \T) = \text{ind}(\leftthreetimes I - \T) \leq 0\) and thus \(\leftthreetimes_0 I - \T\) is an upper semi-Weyl operator, which implies that  \(\leftthreetimes_0\notin \mathfrak{S}_{ubw}(\T )\). Therefore, \(\mathfrak{S}_{ubw}(\T ) \subset \mathfrak{S}_{ubf}(\T )\).   
 \end{proof}

The following equivalence theorem will allow us to obtain the calculation of some other spectra.

\begin{Th} \label{T1.7}
 Let $\mathbb{T} \in \mathbb{L }(\mathbb{X})$ verifies
     property $(E_A)$,  if and only if,  $\mathbb{T}$  verifies
     property $(gE_A)$.
     \end{Th}
  \begin{proof}
We assume that $\T$ verifies property $(E_A)$. By Theorem \ref{T1.6}, we have \esuf  $=$ \esuw  and by Lemma \ref{T1.44}, it follows that \esubf $=$ \esubw, which implies that $\mathbb{T}$ verifies property $(gE_A)$.
    
  If $\mathbb{T}$ has the property $(gE_A)$, then by a similar reasoning as in Theorem \ref{T1.6}, we obtain that property $(gE_A)$ implies that \esubw $=$ \esubf, then again, by Lemma \ref{T1.6}, it follows that  \esuf = \esuw. Therefore, $\mathbb{T}$ verifies property $(E_A)$.
\end{proof}

\begin{Example}\label{exam2}
       Let $\mathbb{S}:  \ell^2(\mathbb{N})  \longrightarrow \ell^2(\mathbb{N})$  be an injective quasinilpotent operator which is not nilpotent. Define the operator \(\T\) on the Banach space \(\ell^2(\mathbb{N}) \oplus \ell^2(\mathbb{N}) \) by \(\T = I \oplus \mathbb{S}\), where \(I\) is the identity operator. We have \(\mathfrak{S}_{a}(\T) = \{0,1\}\), \(\mathfrak{S}_{ubf}(\T) = \{0\}\) and \(\mathfrak{S}_{uf}(\T) = \{0,1\}\). Since \(\mathfrak{S}_{a}(\T) \subseteq \mathfrak{S}_{uw}(\T) \subseteq \mathfrak{S}_{uf}(\T)\), it turns out that \(\mathfrak{S}_{uf}(\T) = \mathfrak{S}_{uw}(\T)\), which implies that  \(\T\)  verifies property $(E_A)$. Also, by Lemma \ref{T1.44}, we have \(\mathfrak{S}_{ubw}(\T)  = \{0\}\).
\end{Example}

\begin{Example}
Let \(\T\) be the unilateral right shift on  $\ell^2(\mathbb{N})$. It is a well-known fact that \(\mathfrak{S}_{uf}(\T) =\mathfrak{S}_{ubw}(\T) = \mathfrak{S}_{uw}(\T) = \Gamma(0,1)\), so \(\T\) verifies properties \((E_A)\). Therefore, by Lemma \ref{T1.44}, \(\mathfrak{S}_{ubf}(\T) =  \Gamma(0,1)\).
 \end{Example}

\section{Sufficient conditions that implies property $(E_A)$.}

In reference \cite{EJE}, property $(bz)$ was studied in depth; in that study, it is seen that an operator $\mathbb{T}$ verifies property $(bz)$, if and only if, $\mathfrak{S}_{uf}(\mathbb{T})=\mathfrak{S}_{ub}(\mathbb{T})$, so property $(bz)$ implies property $(E_A)$, because \(\mathfrak{S}_{uf}(\T ) \subseteq \mathfrak{S}_{uw}(\T ) \subseteq \mathfrak{S}_{ub}(\T )\). A topological condition equivalent to property $(bz)$ is that, $int(\delta_{uf}(\mathbb{T}))= \emptyset$. We propose a new topological condition but in view of the connection.

\begin{Th}
Let be $\mathbb{T} \in \mathbb{L }(\mathbb{X})$. If $\rho_{uf}(\mathbb{T})$ is connected, then  $\mathbb{T}$  verifies property $(bz)$. 
\end{Th}
  \begin{proof}
 It is clear that $\rho(\mathbb{T}) \subseteq\rho_{uf}(\mathbb{T})$, also, $\mathbb{T}$ verifies the SVEP for all  $\leftthreetimes \in \rho(\mathbb{T})$, hence, there is a  $\leftthreetimes \in \Omega$ for which $\mathbb{T}$ has the SVEP, where  
 $ \Omega$ is the unique component of  $\rho_{uf}(\mathbb{T})$. Hence, by \cite[Theorem 3.36]{Aiena2} it turns out that $\mathbb{T}$ has the SVEP for all $\leftthreetimes \in \Omega$. Now, let $\leftthreetimes \notin \mathfrak{S}_{uf}(\mathbb{T})$, so $\leftthreetimes \in \rho_{uf}(\mathbb{T})$, and so $\leftthreetimes \in  \Omega$, whereby $\mathbb{T}$ has the SVEP at $\leftthreetimes$, then  $p(\leftthreetimes I-\mathbb{T})<+\infty$, hence $\leftthreetimes \notin \mathfrak{S}_{ub}(\mathbb{T})$. Thus, we deduce that 
  $\mathbb{T}$  verifies property $(bz)$.      
  \end{proof}

\begin{Cor}
Let be $\mathbb{T} \in \mathbb{L }(\mathbb{X})$. If $\rho_{uf}(\mathbb{T})$ is connected, then  $\mathbb{T}$  verifies property $(E_A)$. 
\end{Cor}

All characterizations and properties of the property  $(bz)$ apply to the property  $(E_A)$, since, we pointed out that property $(bz)$ implies property $(E_A)$, but the reverse is not true, as the following example indicates.

\begin{Example}  
Let \(\mathbb{L}\) be the Unilateral Left Shift on \(\ell^2(\N)\), and \(\mathbb{R}\) be the Unilateral Right Shift on \(\ell^2(\N)\). Define $\mathbb{T}:= \mathbb{L} \oplus \mathbb{R}$, which has $ \mathfrak{S}_{uf}(\T ) = \Gamma(0,1)$, $ \mathfrak{S}_{ub}(\T ) = D(0,1)$, see \cite{Ouidren222} for details. Note that, $\alpha(\mathbb{T})=1=\beta(\mathbb{T})$, whereby $0\notin \mathfrak{S}_{uw}(\mathbb{T})$. It is known that \(\mathbb{L}\) does no have SVEP at $0$, thus \(\mathbb{T}\) does no have SVEP at $0$, since is an upper semi-Fredholm operator, so $0\in \mathfrak{S}_{ub}(\mathbb{T})$. Hence, $\mathfrak{S}_{ub}(\mathbb{T}) \neq \mathfrak{S}_{uw}(\mathbb{T})$. This showed that $\mathbb{T}$ does not verify property $(bz)$. On the other hand, as $\mathbb{R^*} =\mathbb{L}$, so  $ind(\mathbb{L}) = - ind(\mathbb{R})$, also, $ind(\mathbb{T})= ind(\mathbb{L}) + ind(\mathbb{R})$ and $ \mathfrak{S}_{uf}(\T ) = \Gamma(0,1)$, then $ \mathfrak{S}_{uw}(\T ) =  \mathfrak{S}_{uf}(\T )$. Therefore, $\mathbb{T}$ verify property $(E_A)$.

\end{Example}

With the help of SVEP, we prove a sufficient condition for an operator $\mathbb{T} \in \mathbb{L }(\mathbb{X})$ to verify the property $(E_A)$.

\begin{Th}\label{T4.2}
Let be $\mathbb{T} \in \mathbb{L }(\mathbb{X})$. If  $Acc(\mathfrak{S}_{a}(\mathbb{T})) \subseteq   \mathfrak{S}_{uf}(\mathbb{T})$, then $\mathbb{T}$ verifies property $(E_A)$.
\end{Th} 
\begin{proof}
  
Let $\leftthreetimes \notin$ \esuf $ \cup  Acc(\mathfrak{S}_{a}(\mathbb{T}))$,
then $\leftthreetimes \in$ (\esuf)$^{c}$ $\cap $ \isoesa,  which implies that $(\leftthreetimes I-\mathbb{T})$ an upper semi-Fredholm operator and that $\mathbb{T}$ has SVEP at $\leftthreetimes$, meaning that $p(\leftthreetimes I-\mathbb{T})<+\infty$. Thus, $\leftthreetimes \notin$ \esuw. Therefore, $\mathfrak{S}_{uw}(\mathbb{T}) \subseteq   \mathfrak{S}_{uf}(\mathbb{T})\cup Acc(\mathfrak{S}_{a}(\mathbb{T}))$. The hypothesis implies that $\mathfrak{S}_{uw}(\mathbb{T}) =   \mathfrak{S}_{uf}(\mathbb{T})$. So that,  $\mathbb{T}$ verifies property $(E_A)$. 
\end{proof}

In view of Theorem \ref{T4.2}. If for $\mathbb{T} \in \mathbb{L }(\mathbb{X})$, it turn out that $\mathbb{T^*}$, or $\mathbb{T}$ does not verify the property $(E_A)$, then $Acc(\mathfrak{S}(\mathbb{T})) \neq \emptyset$.
\begin{Cor}
Let be $\mathbb{T} \in \mathbb{L }(\mathbb{X})$. Then: 
\begin{enumerate}[\upshape (i)]

    \item   If  $Acc(\mathfrak{S}_{s}(\mathbb{T})) \subseteq   \mathfrak{S}_{lf}(\mathbb{T})$, then $\mathbb{T^*}$ verifies property $(E_A)$.

    \item   If  $Acc(\mathfrak{S}_{s}(\mathbb{T})) = \emptyset$, then $\mathbb{T^*}$ verifies property $(E_A)$.
    
     \item  If  $Acc(\mathfrak{S}_{a}(\mathbb{T})) = \emptyset$, then  $\mathbb{T}$ verifies
     property $(E_A)$.

      \item  If  $Acc(\mathfrak{S}(\mathbb{T})) = \emptyset$, then  $\mathbb{T}$ and  $\mathbb{T^*}$ verify the property $(E_A)$.
\end{enumerate}

\end{Cor} 
 
\begin{Example}
   Operators that have a power of finite rank are algebraic; the spectrum of an algebraic operator is a finite set. Thus, the set of accumulated points of the spectrum is empty. Therefore, these operators verify the property  $(E_A)$.
\end{Example}

The following result, for $\mathbb{T} \in \mathbb{L }(\mathbb{X})$, indicates the difference between $\mathfrak{S}_{uf}(\mathbb{T})$ and $ \mathfrak{S}_{ubf}(\mathbb{T})$, if the operator verifies property $(E_A)$.
 
\begin{Th}
Let be $\mathbb{T} \in \mathbb{L }(\mathbb{X})$. If $\mathbb{T}$  verifies property $(E_A)$, then $$\mathfrak{S}_{uf}(\mathbb{T}) =    \mathfrak{S}_{ubf}(\mathbb{T})\cup Acc(\mathfrak{S}_{uf}(\mathbb{T})).$$
\end{Th}
\begin{proof}
   By  \cite[Theorem 3.55]{Aiena}, it turns out that $\mathfrak{S}_{uw}(\mathbb{T}) =    \mathfrak{S}_{ubw}(\mathbb{T})\cup Acc(\mathfrak{S}_{uw}  (\mathbb{T}))$ since \(\T\) satisfies \((E_A)\) so by Theorem 1.6 and Theorem 1.5, we have that 
   \(\mathfrak{S}_{uf}(\T) = \mathfrak{S}_{uw}(\T)\) and \(\mathfrak{S}_{ubw}(\T) = \mathfrak{S}_{ubf}(\T)\), and so we get the result. 
\end{proof}
   \begin{Cor}
Let be $\mathbb{T} \in \mathbb{L }(\mathbb{X})$. If    $Acc(\mathfrak{S}_{a}(\mathbb{T})) \subseteq   \mathfrak{S}_{uf}(\mathbb{T})$, then $$\mathfrak{S}_{uf}(\mathbb{T}) =    \mathfrak{S}_{ubf}(\mathbb{T})\cup Acc(\mathfrak{S}_{uf}(\mathbb{T})).$$
\end{Cor}

  The $\mathfrak{S}_{a}(\mathbb{T})$ does not necessarily coincide with $  \mathfrak{S}_{uf}(\mathbb{T})$; this happens for a large number of operators, for example, the Unilateral Left Shift operator has $\mathfrak{S}_{a}(\mathbb{L})=D(0,1)\neq \mathfrak{S}_{uf}(\mathbb{L}) =\Gamma(0,1)$, see Example \ref{e3.5}. It is guaranteed, for $\mathbb{T} \in \mathbb{L }(\mathbb{X})$, that if $\mathfrak{S}_{a}(\mathbb{T}) = \mathfrak{S}_{uf}(\mathbb{T})$, so the operator verifies the property $(bz)$, and hence $(E_A)$. We see this follow.
  
\begin{Th}\label{c11}
 Let be $\mathbb{T} \in \mathbb{L }(\mathbb{X})$.  If   $\mathfrak{S}_{a}(\mathbb{T}) = \mathfrak{S}_{uf}(\mathbb{T})$, then $\mathbb{T}$  verifies  property $(bz)$.
\end{Th}
\begin{proof}
    It is a well known fact that \(\mathfrak{S}_{uf}(\T) \subseteq \mathfrak{S}_{ub}(\T) \subseteq \mathfrak{S}_{a}(\T) \), so using the hypothesis, we have that \( \mathfrak{S}_{uf}(\T) = \mathfrak{S}_{ub}(\T)\), which implies that   \(\T\) satisfies the  property \((bz)\).
\end{proof}

\begin{Cor}\label{c1122}
 Let be $\mathbb{T} \in \mathbb{L }(\mathbb{X})$.  If   $\mathfrak{S}_{a}(\mathbb{T}) = \mathfrak{S}_{uf}(\mathbb{T})$, then $\mathbb{T}$  verifies  property $(E_A)$.
\end{Cor}

  \begin{Example} It is known that for $R$, the Unilateral Right Shift operator, it turn out that  $\mathfrak{S}_{a}(R) = \Gamma(0,1)$. Then by Example \ref{e3.5} it is obtained that $\mathfrak{S}_{a}(R)= \mathfrak{S}_{uf}(R)$. Also, if $\mathbb{T}$ is the operator indent $I$ in $\mathbb{L }(\mathbb{X})$, then $\mathfrak{S}_{a}(\mathbb{T})= \mathfrak{S}_{uf}(\mathbb{T}) = \{1\}$. Similarly, it is happen for the operator defined as;  $P(x_1,x_2,x_3,...)=(\frac{x_2}{2}, \frac{x_3}{3},...), \text{  for all } (x_n)\in \ell^2(\mathbb{N}).$ In this case $\mathfrak{S}_{a}(P)= \{0\}$. Also, another example is the Cesàro operator defined in Example \ref{ejem11}. Therefore, by corollary \ref{c1122}, all these operators verify property $(E_A)$.
\end{Example}

\section{Conclusions}

The property   $(E_A)$   collects in a way operators, with closed range, which are not injective and have a negative Weyl's index. But if an operator, we say  $\mathbb{T} \in \mathbb{L }(\mathbb{X})$,   does not verifies the property $(E_A)$, then:
\begin{enumerate}[\upshape (i)]

\item  $ \mathfrak{S}_{uf}(\mathbb{T}) \neq\mathfrak{S}_{uw}(\mathbb{T}).$

\item  $ \mathfrak{S}_{ubf}(\mathbb{T}) \neq \mathfrak{S}_{ubw}(\mathbb{T}).$

    \item $\mathfrak{S}_{a}(\mathbb{T}) \neq \mathfrak{S}_{uf}(\mathbb{T})$.

    \item $\rho_{uf}(\mathbb{T})$ is not connected.

  \item $Acc(\mathfrak{S}_{a}(\mathbb{T})) \neq \emptyset.$ 

  \item  $Acc(\mathfrak{S}_{a}(\mathbb{T})) \cap   \rho_{uf}(\mathbb{T})\neq \emptyset.$

   \item  $\mathbb{T}$ does not verifies the property $(bz)$.
    
\end{enumerate}


\begin{thebibliography}{9}

\bibitem{Aiena} Aiena, P. \emph{Fredholm and local spectral theory II with applications to Weyl-type Theorems}. {Springer}, \textbf{2018}.

\bibitem{AN} Aiena, P. {Quasi-Fredholm operators and localized $\mathcal{SVEP}$}. \emph{Acta Mathematica Scientia}, {73}, 251--263, \textbf{2007}.

\bibitem{Aiena2} Aiena, P. 
\emph{Fredholm and Local Spectral Theory, with Applications to Multipliers}. Kluwer Academic Publishers: Dordrecht, The Netherlands, \textbf{2004}.

\bibitem{AAG} Aiena, P., Aponte, E., and Guill\'en, J. {The Zariouh's property $(gaz)$ through localized SVEP}. \emph{Mat. Vesnik}, 72(4), 314--326, \textbf{2020}.

\bibitem{AGP2} Aiena, P., Guill{\'e}n, J., and Pe\~na, P. {Property (gb) through local spectral theory}. \emph{Math. Proc. R. Ir. Acad.}, 114(1), 1--15, \textbf{2014}.

\bibitem{AGP1} Aiena, P., Guill{\'e}n, J., and Pe\~na, P. {Localized SVEP, property $(b)$ and property $(ab)$.}  \emph{Mediterranean Journal of Mathematics}, 10(4), 1965--1978, \textbf{2013}.

\bibitem{AJQV} Aponte, E., Jayanthi, N., Quiroz, D., and Vasanthakumar, P. {Tensor Product of Operators Satisfying Zariouh’s Property $(gaz)$, and Stability under Perturbations}. \emph{Axioms}, {11}(5), 225, \textbf{2022}.

\bibitem{uuu} Aponte, E., Mac\'ias, J., Sanabria, J., and Soto, J. {Further characterizations of property $(V_{\Pi})$ and some applications}. \emph{Proyecciones},  39(6),  1435--1456, \textbf{2020}.

\bibitem{EJE} Aponte, E., Soto, J., and Rosas E. {Study of the property (bz) using local spectral theory methods}. \emph{Arab Journal of Basic and Applied Sciences}, {30}(1), 665--674, \textbf{2023}.

\bibitem{Tensor2022} Aponte, E., Vasanthakumar, P., and Jayanthi, N. {Property $(Bv)$ and Tensor Product}. \emph{Symmetry}, {14}(10), 2011, \textbf{2022}.

\bibitem{ouidren2} Ben Ouidren, K. and Zariouh, H. {New approach to a-Weyl's theorem through localized SVEP and Riesz-type perturbations}. \emph{Linear and Multilinear Algebra}, 70(17), 3231--3247, \textbf{2022}.

\bibitem{new2}
{Ben Ouidren, K. and Zariouh, H. {Extended Rakocevic property.} \emph{Funct. Anal. Approx. Comput.}, 13(1), 27--34, \textbf{(2021)}.}

\bibitem{Ouidren222} Ben Ouidren, K. and Zariouh, H. {New approach to a-Weyl's theorem and some preservation results.} \emph{Rend. Circ. Mat. Palermo Series 2}, 70, 819--833, \textbf{2021}.

\bibitem{Chen} Chen, Y., Chi, Y., Fan, J., and Ma, C. {Spectral methods for data science: A statistical perspective}. \emph{Foundations and Trends in Machine Learning}, 14(5), 566--806, \textbf{2021}.

\bibitem{Cob} Coburn, L. A.   Weyl's theorem for nonnormal operators. \emph{Mich. Math. J.}, {13}, 285--288, \textbf{1966}.

\bibitem{Han} Djordjevi\'{c}, S.V. and Han, Y.M. {Browder's theorem and spectral continuity.} \emph{Glasgow Math. J.}, 42(3), 479--486, \textbf{2000}.

\bibitem{Fi} Finch, J. K. {The single valued extension property on a Banach space.} \emph{Pacific J. Math.}, {58}(1), 61--69, \textbf{1975}.

\bibitem{mg} González, M. {The fine spectrum of the Cesàro operator in  $\ell_p$ ($1<p<\infty $)}. \emph{Arch. Math.}, 44, 355--358,  
 \textbf{1985}.
 
\bibitem{R.Harte} Harte, R. and Lee, W.Y. {Another note on Weyl's theorem.} \emph{Trans. Amer. Math. Soc.}, 349(5), 2115--2124, \textbf{1997}.

\bibitem{roa} Rhoades, B.E. and Neumann, M. M. {Spectra of some Hausdorff operators}. \emph{Acta Sci. Math. (Szeged)}, 32, 91--100, \textbf{1971}.

\bibitem{SANA0b} Sanabria, J., V\'asquez, L., Carpintero, C., Rosas, E., and Garc\'ia, O. {On strong variations of Weyl type theorems}. \emph{Acta Math. Univ. Comen. (N.S.)}, {86}(2), 345--356, \textbf{2017}.

\bibitem{Saul} Saul, L. K., Weinberger, K. Q., Ham, J. H., Sha, F., and Lee, D. D. {Spectral methods for dimensionality reduction}. \emph{Semi-supervised learning}, {3}, \textbf{(1966)}.

\bibitem{weyl} Weyl, H. {{\"U}ber beschr{\"a}nkte quadratiche Formen, deren Differenz vollsteig ist.} \emph{Rend. Circ. Mat. Palermo}, 27(1), 373--92, \textbf{1909}.




 
 
\end{thebibliography}
\end{document}